\documentclass[12pt,twoside]{amsart}
\usepackage{amssymb}
\usepackage{verbatim}
\usepackage{amsmath}
\usepackage{bm}
\usepackage{a4wide}
\usepackage[latin1]{inputenc}
\usepackage[T1]{fontenc}
\usepackage{times}
\usepackage{amssymb,latexsym}
\usepackage{enumerate}
\usepackage{color}

\makeatletter
\newcommand{\sumprime}{\if@display\sideset{}{'}\sum%
            \else\sum'\fi}
\makeatother

\begin{document}

\numberwithin{equation}{section}

\newtheorem{theorem}{Theorem}[section]
\newtheorem{proposition}[theorem]{Proposition}
\newtheorem{conjecture}[theorem]{Conjecture}
\def\theconjecture{\unskip}
\newtheorem{corollary}[theorem]{Corollary}
\newtheorem{lemma}[theorem]{Lemma}
\newtheorem{observation}[theorem]{Observation}
\newtheorem{definition}{Definition}
\numberwithin{definition}{section} 
\newtheorem{remark}{Remark}
\def\theremark{\unskip}
\newtheorem{kl}{Key Lemma}
\def\thekl{\unskip}
\newtheorem{question}{Question}
\def\thequestion{\unskip}
\newtheorem{example}{Example}
\def\theexample{\unskip}
\newtheorem{problem}{Problem}

\thanks{}

\address{Department of Mathematical Sciences, Fudan University, Shanghai, 200433, China}

\email{boychen@fudan.edu.cn}

\title{An $L^2$ Hartogs-type extension theorem for unbounded domains}
\author{Bo-Yong Chen}

\date{}

\begin{abstract}
In this note,  we prove an $L^2$ Hartogs-type extension theorem for unbounded domains. 
\end{abstract}

\maketitle

\section{Introduction}

In the landmark paper \cite{Hartogs},  Hartogs proved the following celebrated result.

\begin{theorem}[Hartogs,  1906]\label{th:Hartogs}
Let $\Omega$ be a domain in $\mathbb C^n$ with $n\ge 2$ and $E$ a compact subset in $\Omega$.  If $\Omega\backslash E$ is connected,  then every holomorphic function on $\Omega\backslash E$ can be extended holomorphically to $\Omega$.
\end{theorem}

There are at least three approaches for the Hartogs extension theorem.  The first one,  which is the original proof of Hartogs,  was completed only recently by Merker-Porton \cite{MP1}; the second one is based on the Bochner-Martinelli formula (cf.  \cite{Bochner3},  \cite{Martinelli}); the third one,  which is the most popular,  is by using the $\bar{\partial}-$method (cf.  \cite{Ehrenpreis}).    Generalizations to complex manifolds and complex spaces are also available (see e.g.,  \cite{CR},  \cite{MP2,Ohsawa,OV}).  We  refer to the  paper of Range \cite{Range} for  a very interesting historical recollection on this topic. 

The Hartogs extension phenomenon for the case when $E$ is an\/ {\it unbounded} closed subset seems to be more involved.   A  classic example in this direction is  the following tube theorem obtained by Bochner \cite{Bochner2}. 

\begin{theorem}[Bochner,  1938]\label{th:tube}
Every holomorphic function defined on the tube $D\times i\mathbb R^n$,  where $D$ is an open set in $\mathbb R^n$,  can be extended holomorphically to $(\text{convex hull of\ }D)\times i\mathbb R^n$. 
\end{theorem}

There are some generalizations  to certain "tube-like" domains (cf.  \cite{BDS1,BDS2}).  In particular,   \cite{BDS2}  indicates the complexity of  the Hartogs extension phenomenon for unbounded cases.    

 Actually,  Bochner had proved an $L^2$ version of Theorem \ref{th:tube} in  an earlier paper \cite{Bochner1},  which seems to be less known.

\begin{theorem}[Bochner,  1937]\label{th:Bochner37}
Every $L^2$ holomorphic function defined on  $D\times i\mathbb R^n$ can be extended to an $L^2$ holomorphic function on  $(\text{convex hull of\ }D)\times i\mathbb R^n$. 
\end{theorem}

Motivated by this theorem,  we shall prove the following $L^2$ Hartogs extension theorem.

\begin{theorem}\label{th:Main}
Let $\Omega$ be a domain in $\mathbb C^n$ and $E$ a closed set in $\mathbb C^n$ such that
\begin{enumerate}
\item[$(1)$] there exists $r>0$ such that $E_r:=\{z\in \mathbb C^n: d(z,E)\le r\}\subset \Omega$;
\item[$(2)$] there exist an affine-linear subspace $H\subset \mathbb R^{2n}=\mathbb C^n$ of real codimension $\ge 3$ and a number $R>0$ such that 
$$
E\subset H_R:=\{z\in \mathbb C^n: d(z,H)<R\};
$$
\item[$(3)$] $\Omega\backslash E$ is connected.
\end{enumerate}
Then every $L^2$ holomorphic function defined on $\Omega\backslash E$ can be extended to an $L^2$ holomorphic function on $\Omega$.
\end{theorem} 

Theorem \ref{th:Hartogs} follows directly from Theorem \ref{th:Main}.  To see this,  first take a domain $\Omega':E\subset \Omega'\subset\subset \Omega$ and $H=\{0\}$,  then apply Theorem \ref{th:Main} to the pair $(\Omega',E)$.  It is also easy to see that Theorem \ref{th:Main} contains some special cases of 
Theorem \ref{th:Bochner37},  e.g.,  $D=D'\backslash E$ where $D'$ is a  convex domain in $\mathbb R^n$ and $E$ is a compact subset in $D'$.   

The proof of Theorem \ref{th:Main} relies heavily on the classic Hardy inequality,  which also reveals that the basic reason for the Hartogs extension phenomenon is nothing but the non-parabolicity of $\mathbb C^n=\mathbb R^{2n}$ when $n\ge 2$.

\section{An $L^2-$estimate for the $\bar{\partial}-$equation in $\mathbb C^n$}

Let $\nabla$ and $\Delta$ denote the standard gradient and real Laplacian.  We shall prove the following

\begin{theorem}[compare \cite{Chen}]\label{th:L2}
Suppose that there exists a  measurable function $\omega\ge 0$ on $\mathbb C^n$ such that 
\begin{equation}\label{eq:Hardy}
\int_{\mathbb C^n} \phi^2 \omega\le \int_{\mathbb C^n} |\nabla \phi|^2
\end{equation}
holds for any real-valued smooth function $\phi$ with compact support in $\mathbb C^n$.  Then for any  $\bar{\partial}-$closed $(0,q)-$form $v$ on $\mathbb C^n$ with $\int_{\mathbb C^n} |v|^2<\infty$ and $\int_{\mathbb C^n} |v|^2/\omega<\infty$,  there exists a $(0,q-1)-$form $u$ on $\mathbb C^n$ such that  $\bar{\partial}u=v$ and
$$
\int_{\mathbb C^n} |u|^2 \le 4\int_{\mathbb C^n} |v|^2/\omega.
$$ 
\end{theorem}

\begin{remark}
 \eqref{eq:Hardy} is usually called a Hardy-type inequality in literature (the special case when $\omega=\frac{(n-1)^2}{|z|^2}$ is the standard Hardy inequality),  which is of particular importance in real analysis and partial differential equations. 
\end{remark}

\begin{proof}
Let $\phi=\phi_1+i\phi_2$,  where $\phi_1,\phi_2$ are real-valued smooth functions with compact supports in $\mathbb C^n$.  By \eqref{eq:Hardy} we have 
\begin{eqnarray}\label{eq:dbar_1}
\int_{\mathbb C^n} |\phi|^2 \omega & = & \int_{\mathbb C^n} \phi_1^2\, \omega + \int_{\mathbb C^n} \phi_2^2\, \omega \nonumber\\
& \le & \int_{\mathbb C^n} \left(|\nabla \phi_1|^2 + |\nabla \phi_2|^2 \right)\nonumber\\
& = & - \int_{\mathbb C^n}\left(\phi_1\Delta\phi_1+\phi_2\Delta \phi_2\right)\ \ \ (\text{Stokes formula})\nonumber\\
& = & - \int_{\mathbb C^n}\phi\, \overline{\Delta \phi},
\end{eqnarray}
where the last equality follows from
$$
\int_{\mathbb C^n} \phi_1\Delta\phi_2=-\int_{\mathbb C^n} \nabla \phi_1\cdot \nabla \phi_2= 
\int_{\mathbb C^n} \phi_2\Delta\phi_1.
$$
Let $\mathcal D_{(0,q)}(\mathbb C^n)$ denote the set of smooth $(0,q)-$forms with compact supports in $\mathbb C^n$ and $L^2_{(0,q)}(\mathbb C^n)$ the completion of $\mathcal D_{(0,q)}(\mathbb C^n)$ with respect to the standard $L^2$ norm.  Let $\vartheta$ be the formal adjoint of $\bar{\partial}$ and $\Box:=\bar{\partial}\vartheta+\vartheta\bar{\partial}$ the complex Laplacian.  In what follows we shall use  standard terminologies in H\"ormander's classic book \cite{Hormander}.  For any $u=\sum'_{I,J}u_{IJ} dz_I\wedge d\bar{z}_J\in \mathcal D_{(0,q)}(\mathbb C^n)$,  we infer from \eqref{eq:dbar_1} that
\begin{eqnarray}\label{eq:dbar_3}
\int_{\mathbb C^n} |u|^2 \omega & \le & -{\sum_{I,J}}' \int_{\mathbb C^n} u_{IJ}\overline{\Delta u_{IJ}}\nonumber\\
& = & -4 (u,\Box u)\nonumber\\
& = & 4\left(\|\bar{\partial} u\|^2+\|\vartheta u\|^2 \right).
\end{eqnarray}
Note  that $\bar{\partial}: L^2_{(0,q)}(\mathbb C^n)\rightarrow L^2_{(0,q+1)}(\mathbb C^n)$ gives a  densely defined and closed operator.  Let $\bar{\partial}^\ast$ be the Hilbert space adjoint of $\bar{\partial}$.  It is well-known that $\mathcal D_{(0,q)}(\mathbb C^n)$ lies dense in $\mathrm{Dom\,}\bar{\partial}\cap \mathrm{Dom\,}\bar{\partial}^\ast$ under the following graph norm:
$$
u\mapsto \|u\| + \|\bar{\partial}u\|+\|\bar{\partial}^\ast u\|.
$$  
Thus for any $u\in \mathrm{Dom\,}\bar{\partial}\cap \mathrm{Dom\,}\bar{\partial}^\ast$ there exists a sequence $\{u_j\}\subset \mathcal D_{(0,q)}(\mathbb C^n)$ such that $u_j\rightarrow u$ under the graph norm.  Replace $\{u_j\}$ by a subsequence,  we may assume furthermore that $u_j\rightarrow u$ a.e.  in $\mathbb C^n$.  It follows from Fatou's lemma that 
\begin{eqnarray}\label{eq:dbar_2}
\int_{\mathbb C^n} |u|^2 \omega & \le & \liminf_{j\rightarrow \infty} \int_{\mathbb C^n} |u_j|^2 \omega\nonumber\\
& \le & 4 \liminf_{j\rightarrow \infty}\left(\|\bar{\partial} u_j\|^2+ \|\vartheta u_j\|^2\right)\nonumber\\
& = & 4 \left(\|\bar{\partial} u\|^2+ \|\bar{\partial}^\ast u\|^2\right).
\end{eqnarray}
Now we can apply the standard duality argument.  Consider the mapping
$$
T: \bar{\partial}^\ast w\mapsto (w,v),\ \ \ w\in \mathrm{Dom\,}\bar{\partial}^\ast\cap \mathrm{Ker\,}\bar{\partial}.
$$
The Cauchy-Schwarz inequality gives
\begin{eqnarray}\label{eq:dbar_3}
|(w,v)|^2 & \le & \int_{\mathbb C^n} |w|^2 \omega\cdot \int_{\mathbb C^n} |v|^2/\omega\nonumber\\
& \le & 4\|\bar{\partial}^\ast w\|^2 \cdot \int_{\mathbb C^n} |v|^2/\omega
\end{eqnarray}
in view of \eqref{eq:dbar_2}.  Thus $T$ is a well-defined continuous linear functional on $\mathrm{Dom\,}\bar{\partial}^\ast\cap \mathrm{Ker\,}\bar{\partial}$ with
\begin{equation}\label{eq:dbar_4}
\|T\|^2 \le 4\int_{\mathbb C^n} |v|^2/\omega.
\end{equation}
Since $v\in \mathrm{Ker\,}\bar{\partial}$,  we have $(w,v)=0$ for all $w\bot \mathrm{Ker\,}\bar{\partial}$,  so that $T$ extends to a continuous linear functional on the range of $\bar{\partial}^\ast$ which still satisfies \eqref{eq:dbar_4}.  The Hahn-Banach theorem combined with the Riesz representation theorem gives a unique $u\in L^2_{(0,q-1)}(\mathbb C^n)$ such that $\|u\|\le 2 \|v/\sqrt{\omega}\|$ and 
$$
(w,v)=(\bar{\partial}^\ast w,u),\ \ \ w\in \mathrm{Dom\,}\bar{\partial}^\ast,
$$
i.e.,  $\bar{\partial} u=v$ holds in the sense of distributions.
\end{proof}

Let us recall the following classic result.

\begin{lemma}[Hardy inequality]\label{lm:Hardy}
Let $H\subset \mathbb R^{N}$ be an affine-linear subspace with codimension   $m\ge 3$.   Define $d_H=d(\cdot,H)$.  Then we have
$$
\frac{(m-2)^2}4 \int_{\mathbb R^{N}} \phi^2/d_H^2 \le \int_{\mathbb R^{N}} |\nabla \phi|^2
$$
for any smooth real-valued function with compact support in $\mathbb R^{N}$.
\end{lemma}

\begin{proof}
For the sake of completeness,  we still provide a proof here.  We may assume $H=\{x'=0\}$ where $x'=(x_1,\cdots,x_m)$.  Then we have $d_H(x)=|x'|$ and the function $\psi(x)=\psi(x')=-|x'|^{2-m}$ is subharmonic on $\mathbb R^m$,  hence is subharmonic on $\mathbb R^N$ (in particular,  $\mathbb R^N$ is non-parabolic).  Thus
$$
0\le \int_{\mathbb R^N} \phi^2\cdot \frac{\Delta \psi}{-\psi}=- \int_{\mathbb R^N} \nabla \psi \cdot \nabla\left(\frac{\phi^2}{-\psi}\right)=2 \int_{\mathbb R^N} \frac{\phi}{\psi}\cdot\nabla \psi\cdot\nabla \phi-\int_{\mathbb R^N} \frac{\phi^2}{\psi^2}\cdot |\nabla \psi|^2,
$$
so that
\begin{eqnarray*}
\int_{\mathbb R^N} \frac{\phi^2}{\psi^2}\cdot |\nabla \psi|^2 & \le & -2\int_{\mathbb R^N} \frac{\phi}{\psi}\cdot\nabla \psi\cdot\nabla \phi\\
& \le & \frac12 \int_{\mathbb R^N} \frac{\phi^2}{\psi^2}\cdot |\nabla \psi|^2 + 2\int_{\mathbb R^N} |\nabla \phi|^2,
\end{eqnarray*}
that is
$$
\int_{\mathbb R^N} \frac{\phi^2}{\psi^2}\cdot |\nabla \psi|^2\le 4\int_{\mathbb R^N} |\nabla \phi|^2.  
$$
On the other hand,  a straightforward calculation gives $|\nabla \psi|^2/\psi^2=(m-2)^2 |x'|^{-2}$,  hence we are done.
\end{proof}

Theorem \ref{th:L2} combined with Lemma \ref{lm:Hardy} immediately yields

\begin{corollary}\label{cor:L2}
Let $H\subset \mathbb R^{2n}=\mathbb C^n$ be an affine-linear subspace with codimension   $m\ge 3$.  Then for any  $\bar{\partial}-$closed $(0,q)-$form $v$ on $\mathbb C^n$ with $\int_{\mathbb C^n} |v|^2<\infty$ and $\int_{\mathbb C^n} |v|^2 d_H^2<\infty$,  there exists a $(0,q-1)-$form $u$ on $\mathbb C^n$ such that  $\bar{\partial}u=v$ and
$$
\int_{\mathbb C^n} |u|^2 \le \frac{16}{(m-2)^2} \int_{\mathbb C^n} |v|^2 d_H^2.
$$ 
\end{corollary}

\section{Proof of Theorem \ref{th:Main}}
Set $d_E(z):=d(z,E)$ and $d_H(z):=d(z,H)$.  Choose a smooth function $\chi:\mathbb R\rightarrow [0,1]$ such that $\chi|_{(-\infty,1/2]}=0$ and $\chi|_{[1,\infty)}=1$.  Given $f\in A^2(\Omega\backslash E):=L^2\cap \mathcal O(\Omega\backslash E)$,  define $v:=\bar{\partial}\{\chi(d_E/r)f\}$ on $\Omega\backslash E$ and $v=0$ on $E\cup(\mathbb C^n\backslash \Omega)$.  Clearly,  $v$ is a smooth $\bar{\partial}-$closed $(0,1)-$form on $\mathbb C^n$.  Moreover,  since $|\nabla d_E|\le 1$ a.e. ,  we have
$$
\int_{\mathbb C^n} |v|^2\le \frac{\sup|\chi'|^2}{r^2}\cdot \int_{r/2\le d_E\le r} |f|^2\le \frac{\sup|\chi'|^2}{r^2}\cdot \int_{\Omega\backslash E} |f|^2<\infty,
$$
and since $E\subset H_R$,  it follows that
\begin{eqnarray*}
\int_{\mathbb C^n} |v|^2 d_H^2 & \le & \frac{\sup|\chi'|^2}{r^2}\cdot \int_{r/2\le d_E\le r} |f|^2d_H^2\\
& \le &  \frac{(R+r)^2}{r^2}\cdot \sup|\chi'|^2\cdot  \int_{\Omega\backslash E} |f|^2<\infty.
\end{eqnarray*}
Thanks to Corollary \ref{cor:L2},  we obtain a solution of $\bar{\partial}u=v$ which satisfies
$$
\int_{\mathbb C^n} |u|^2 \lesssim \int_{\Omega\backslash E} |f|^2.
$$
Since $\mathrm{supp\,}v\subset E_r$,  we conclude that $u\in \mathcal O(\mathbb C^n\backslash E_r)$.  Also,  since $H_{R+r}\subsetneq \mathbb C^n$ is convex,  so there exists a real hyperplane $\mathcal H$ in $\mathbb C^n$ such that $d(\mathcal H,H_{R+r})>1$.  Since $n\ge 2$,  so $\mathcal H$ contains at least one complex line $l$.  Without loss of generality,  we assume  $l=\{z'=0\}$ where $z'=(z_1,\cdots,z_{n-1})$.  Thus the cylinder $\mathcal C:=l\times \mathbb B^{n-1}\subset \mathbb C^n\backslash H_{R+r}\subset \mathbb C^n\backslash E_r$,  where $\mathbb B^{n-1}$ is the unit ball in $\mathbb C^{n-1}$.  Now $u\in A^2(\mathcal C)$,  so $u(z',\cdot)\in A^2(\mathbb C)$ for every $z'\in \mathbb B^{n-1}$,  and it has to vanish in view of the ($L^2$) Liouville theorem.    By the theorem of unique continuation,  $u=0$ in an unbounded component of $\mathbb C^n\backslash E_r$,  which naturally intersects with $\Omega\backslash E$.    Finally,  the function $F:=\chi(d_E/r)f-u$ is holomorphic on $\Omega$ and satisfies $F=f$ on a nonempty open subset in $\Omega\backslash E$.  Since $\Omega\backslash E$ is connected,  it follows that $F=f$ on $\Omega\backslash E$.  Clearly,
$$
\int_\Omega |F|^2 \le 2\int_\Omega |\chi(d_E/r)f|^2 + 2\int_\Omega |u|^2\le \int_{\Omega\backslash E} |f|^2 + 2\int_{\mathbb C^n} |u|^2<\infty.
$$

\end{document}